\documentclass[11pt]{amsart}
\usepackage{amscd}                                            
\usepackage[arrow,matrix]{xy}
\usepackage{graphicx}
\usepackage{hyperref}
\usepackage{comment}
\usepackage{amsmath}
\usepackage{color}
\usepackage{mathabx}
\usepackage{amsmath, latexsym, amssymb}
\numberwithin{equation}{section}
\theoremstyle{plain}
\newtheorem{lemma}{Lemma}[section]
\newtheorem{proposition}[lemma]{Proposition}
\newtheorem{theorem}[lemma]{Theorem}
\newtheorem{corollary}[lemma]{Corollary}

\theoremstyle{definition}
\newtheorem{definition}[lemma]{Definition}
\newtheorem{remark}[lemma]{Remark}
\newtheorem{example}[lemma]{Example}
\usepackage{cancel}

\definecolor{brown}{RGB}{150,100,0}
 
\definecolor{purple}{RGB}{150,0,100}

\definecolor{grey}{RGB}{128,128,128}

\newcommand{\R}{{\mathbb R}}

\newcommand{\N}{{\mathbb N}}

\newcommand{\Di}{{\rm Diff}}
\newcommand{\Om}{{\Omega}}
\newcommand{\om}{{\omega}}
\newcommand{\eps}{{\varepsilon}}  
\newcommand{\codim}{{\rm  codim}}

\newcommand{\Mf}{{\mathfrak M}}
\newcommand{\Nf}{{\mathfrak N}}
\newcommand{\X}{{\mathfrak X}}
\newcommand{\B}{{\mathfrak B}}

\newcommand{\Imm}{{\rm Imm}}

\newcommand{\vol}{{\rm vol}}
\newcommand{\Exp}{{\rm Exp}}

\def\NABLA#1{{\mathop{\nabla\kern-.5ex\lower1ex\hbox{$#1$}}}}
\def\Nabla#1{\nabla\kern-.5ex{}_#1}

\newcommand{\p}{{\partial}}

\newcommand{\la}{\langle}
\newcommand{\ra}{\rangle}

\begin{document}
	\title[ higher dimensional knot spaces]{Formally integrable complex structures on higher dimensional knot spaces}
	
	\author{Domenico Fiorenza}
	\address{
		Dipartimento di Matematica ``Guido Castelnuovo'',
		Universit\`a   di Roma ``La Sapienza",
		Piazzale Aldo Moro 2, 00185 Roma, 
		Italy}
	\email{fiorenza@mat.uniroma1.it} 
	
	\author{H\^ong V\^an L\^e }
	\address{Institute  of Mathematics of the Czech Academy of Sciences,
		Zitna 25, 11567  Praha 1, Czech Republic}
		\email{hvle@math.cas.cz}
		\thanks{Research  of HVL was supported  by  GA\v CR-project 18-00496S and
		 RVO: 67985840}
	\date{\today}


\begin{abstract} Let $S$  be a  compact oriented    finite  dimensional   manifold  and
	$M$    a finite  dimensional Riemannian  manifold, let $\Imm_f(S,M)$    the space of all free immersions  $\varphi:S \to  M$ 
	and let $B^+_{i,f}(S,M)$  the quotient  space  $\Imm_f(S,M)/\Di^+(S)$, where  $\Di^+(S)$ denotes  the  group of   orientation preserving  diffeomorphisms of $S$.     In  this paper   we prove that  if  $M$ admits  a  parallel  $r$-fold  vector cross product  $\chi \in \Om ^r(M, TM)$    and   $\dim  S = r-1$
	then   $B^+_{i,f}(S,M)$     is a  formally K\"ahler  manifold.   This  generalizes       Brylinski's, LeBrun's    and   Verbitsky's results  
	for the  case that $S$ is a   codimension 2 submanifold  in $M$, and  $S = S^1$ or $M$ is a  torsion-free  $G_2$-manifold respectively.   
  \end{abstract}

 \maketitle
  
\section{Introduction}\label{sec:intr}

Let $(M, g)$  be  a  finite  dimensional Riemannian    manifold   with     a  $r$-fold  vector cross  product (VCP for short)
$\chi \in \Om^r (M, TM)$. By definition, this means that  $\chi$  satisfies the following  properties
$$	\la \chi(v_1, \cdots, v_r), v_i  \ra = 0 \text{  for }   1\le i \le r,$$
$$	\la  \chi(v_1, \cdots, v_r), \chi(v_1, \cdots , v_r) \ra = \| v_1 \wedge \cdots \wedge v_r\|^2$$
where $\la   \cdot, \cdot \ra$ denotes  the  inner   product    defined by  $g$ and $\| \cdot \|$ is  the induced metric  on $\Lambda ^r TM$.
For a   VCP $\chi \in \Om^r (M, TM)$  we       associate  the   VCP-form $\varphi_\chi \in \Om ^{r+1}(M)$  as follows \cite[p. 143]{LL2007}
\begin{equation}\label{eq:varphi-chi}
\varphi_\chi  (v_1, \cdots,  v_{r+1}) = \la \chi(v_1, \cdots, v_r), v_{r+1}\ra.
\end{equation}
The notion of a VCP  structure on a manifold was introduced by  Gray \cite{BG1967}, \cite{Gray1969}. In \cite{BG1967}   Brown and Gray  classified  linear VCPs  on an $m$-dimensional real vector space   $V^m$   with positive definite inner   product $g$.  Their results    can be summarized as follows:

(1)  A  1-fold  VCP on $V^m$  exists iff $m = 2n$. It   is defined uniquely by   its  associated VCP-form   which is    a K\"ahler form on $V^{2n}\cong V^n\otimes_{\mathbb R} \mathbb{C}$.

(2)  A $(m-1)$-fold VCP  on $V^m$ is   defined   uniquely by  its associated  VCP-form   which is
the volume form on $(V^m,g)$.

(3) A 2-fold  VCP   exists on $V^m$ iff   $m = 7$  and   it is defined  uniquely
by its associated  VCP-form, which is  the associative 3-form on $\R^7$,  whose  stabilizer is the exceptional Lie group $G_2$.

(4)  A 3-fold  VCP  exists  on  $V^m$  iff $m =8$ and  it is defined   uniquely  by its associated  VCP-form, which is  the Cayley   4-form  on $\R^8$,  whose  stabilizer is  the Lie group 
Spin(7).

It   follows from  Brown-Gray's classification    that a  Riemannian manifold $(M^{m},g)$ admits  a  $r$-fold VCP iff    either  (1)  $r=1$, $m =2n$ and $(M, g)$ is   an  almost Hermitian  manifold and the associated VCP-form is the fundamental 2-form  of the almost Hermitian manifold $(M,g)$,  (2)  $r = m-1$ and  $(M^m, g)$  is an orientable  Riemannian manifold,  (3)  $r =2$,
$m=7$ and $(M^7,g)$ admits  an  associative  3-form  $\varphi^3$,   (4)  $r=3$, $m=8$ and $(M^8, g)$  admits a Cayley  4-form.    
Hence, a Riemannian manifold $(M,g)$ can be  endowed  with a parallel
$r$-fold   VCP  iff  either $(M^m, g)$ is  an orientable    Riemannian  manifold  and  $r = m-1$; or
$m =2 n$,
$(M^{2n}, g)$ is a K\"ahler manifold  and $r =1$; or    $m = 7$  and $(M^7, g)$ is a  torsion-free
$G_2$-manifold and $r =2$; or $m =8$  and  $(M^8, g)$  is a torsion-free   Spin(7)-manifold  and $r =3$. Once more, this result singles out 
K\"ahler  manifolds, torsion-free $G_2$-and Spin(7)-manifolds  as important  classes   of Riemannian manifolds with special holonomy \cite{Joyce2000}. Not unrelatedly,  these classes play  a prominent  role in  calibrated geometry, string theory and $M$-theory \cite{Joyce2007}.

In \cite{LL2007},  motivated by Brylinsky's   results on the   loop spaces of Riemannian  3-manifolds \cite{Brylinski2008},   Lee  and Leung initiated  the study  of moduli spaces $B^+_e(S,M)$ of unparametrized  embedded  oriented submanifolds in  a finite dimensional   Riemannian manifold $(M^m, g)$,  diffeomorphic to  a closed oriented manifold  $S$,
assuming  that $(M^m,g)$ admits an $r$-fold  VCP $\chi$ and  $\dim S = r-1$.  In particular, they proved  that
if  the associated  VCP-form $\varphi_\chi$ is closed  then $\varphi_\chi$ induces a    weak symplectic   form on $B^+_e(S,M)$,  which is   compatible  with  the weak  $L_2$-metric  on $B^+_e(S,M)$, thus defining an almost K\"ahler structure.
Their  result partially extends   Brylinsky's theorem  for $B^+_e(S^1, M^3)$ and LeBrun's extension  to the case   $B^+_{e} (S,  M)$ when $\codim \, S =2$ \cite{LeBrun1993},  stating  that  if $\codim\, S =2$ then  the induced    weak symplectic   form  on $B^+_e(S, M)$ is formally  K\"ahler, i.e., it is such that the  associated  almost complex structure 
has vanishing  Nijenhuis tensor.  Few years    after  Lee-Leung's work,   Verbitsky  proved  that also the  associated  almost complex  structure $J$  on  $B^+_e(S^1, M^7)$ has vanishing  Nijenhuis tensor  if the associated  VCP-form  $\varphi^3$ is parallel and  therefore 
$M^7$  is a  torsion-free    $G_2$-manifold \cite{Verbitsky2010}.
 Almost at the same time, Henrich gave a new proof of Brylinsky's result  by   showing  that   $J$ is parallel  with respect to  the Levi-Civita connection  of the  $L_2$-metric.

 \begin{remark}\label{rem:lempert} The formal integrability  of  the complex structure $J$ on $B^+_{i,f} (S,M)$  
 	does  not  imply  its  strong  integrability,  since  the Newlander-Nirenberg  theorem  does not hold for  infinite  dimensional manifolds \cite{Lempert1993}.  
	Namely, in the infinite dimensional case if the the Nijenhuis tensor $N_J$ of a complex structure $J$ vanishes, then one still has plenty of
local holomorphic functions, but this does not suffice to construct charts with values in an infinite-dimensional Banach or Fr\'echet space. To overcome this difficulty,  in \cite{Lempert1993}
 	Lempert    introduced  the intermediate notion  of  weak integrability of an  almost complex structure
 	and proved   that  the   almost  complex   structure on the  loop  space $B^+_{i,f} (S^1, M^3)$
 	is   weakly integrable.   Since weak integrability implies  formal  integrability,
 	Lempert's result  provides  in particular a refinement of Brylinsky's proof of the formal integrability  of  the  almost complex structure  $J$  on the  space $B^+_{i,f} (S^1,M^3)$.
\end{remark} 	
 \begin{remark} Lempert's   proof   of the weak integrability of  the almost complex structure $J$ on $B^+_{i,f} (S^1, M^3)$  uses  the twistor  construction  of a CR  5-manifold over  a  3-manifold $M^3$  proposed
 	by LeBrun in \cite{LeBrun1984}. This   twistor  construction  has been    generalized further  by LeBrun  (resp. Verbitsky) for their proof  of formal  integrability  of the  almost complex structure $J$
 	on $B^+_{i,f}(S, M)$  in the case that $\mathrm {codim}\,  S = 2$  (resp.  $ \dim S =1$  \&  $\dim M= 7$).   Brylinski's  original proof  of  the formal   integrability  of  the almost complex structure $J$ on $B^+_{i,f} (S^1, M^3)$ uses instead an ingenious computation  of the
 	Nijenhuis tensor  of $J$.
 \end{remark}
 

In this  paper   we show how Henrich's result admits a natural generalization, allowing us to extend Brylinski's  and Verbitsky's results  to all cases
of  Riemannian  manifolds $(M, g)$  with parallel  VCPs.   In particular, this provides a new proof of Verbitsky's result as  well  as a  new proof of  LeBrun's result. Also in the case of $B^+_e(S^1, M^3)$ our proof simplifies Henrich's original argument, as we show that consideration on the parallelism of the almost complex structure $J$ with respect to a torsion-free affine connection $\nabla^\perp$ alone already allows us to conclude that $J$ is formally integrable. One can then derive from this that $J$ is actually also parallel with respect to the Levi-Civita connection on $B^+_e(S^1, M^3)$ and more generally, as we show, with respect to any torsion-free connection of $B^+_e(S, M)$ under consideration.
\medskip

We denote by  $\Di^+ (S)$   the  group of  orientation  preserving diffeomorphisms of $S$  and by $\Imm_f(S, M)$  the set  of  all smooth  free immersions  $\iota \colon  S \to M$, i.e., 
smooth immersions $\iota$ such that the stabilizer  subgroup       of  $\iota$ in $\Di^+ (S)$ is trivial.  For instance, if
$\iota: S \to M$  is   a  somewhere injective immersion, i.e.  there  exists  $x_0 \in S$ such that $\iota ^{-1} (\iota (x_0))=x_0$, then  $\iota \in  \Imm_f (S,M)$ \cite[Lemma 1.4]{CMM1991}. We will denote by $B^+_{i,f} (S, M)$ the quotient  $\Imm_f (S, M)/\Di^+(S)$.  Note that if $S$ is connected, $\Di^+(S)$ is  an  index two subgroup
of the  group $\Di(S)$  of all diffeomorphisms of $S$, and therefore
$B^+_{i,f} (S,M)$ is a  double covering of $B_{i,f}(S,M):=\Imm_f (S, M)/\Di(S) $.  As a side remark we notice that this latter quotient is what is called a {\it shape space}  in computer  vision community
\cite{BBM2013}, \cite{Michor2016}. 

Having introduced this notation, we can state our 

\begin{theorem}[Main Theorem]\label{thm:formalK}  Assume  that  a Riemannian manifold $(M,g)$ admits a parallel $r$-fold
	VCP  $\chi \in \Om^r(M, TM)$.  Let $S$ be a closed  oriented $(r-1)$-dimensional   manifold.
Then the 	space $B^+_{i,f}(S,M)$ has a structure  of a  formally K\"ahler manifold $(B^+_{i,f}(S,M), J, L_2(g), \om^2)$  where  $J$ is an almost complex  structure  with  vanishing Nijenhuis  tensor, $L_2(g)$ is the  weak  $L_2$-metric, and $\om^2$ is the associated      closed  weak symplectic   2-form. Furthermore  $J$  is parallel  w.r.t. to  the  Levi-Civita  connection $\nabla^{LC}$  of the  weak metric  $L_2(g)$.
\end{theorem}

\
{\it Notation  and conventions.}
Through the whole paper  $S$ will always be  a  closed  (i.e., compact without boundary) oriented      finite  dimensional  manifold,  $M$   a    finite dimensional  manifold  and $(M, g)$ a finite dimensional Riemannian manifold, $\Mf$, $\Nf$   smooth  manifolds  modeled  on  convenient vector  spaces, $\X(\Mf)$  the space of smooth (kinematic)  vector fields  on $\Mf$. 
	For a smooth vector bundle $E$ over  a smooth manifold $\Mf$  we denote by $\Gamma (E)$ the space of  all smooth sections of $E$.  For  smooth  manifolds $\Mf$ and $\Nf$  we denote  by $C^\infty(\Mf, \Nf)$  the space  of all smooth mappings from $\Mf$ to $\Nf$.  If  $\Nf = \R$ we     abbreviate $C^\infty(\Mf, \R)$ as $C^\infty(\Mf)$.  For an Euclidean vector bundle $E \to M$ with a fiberwise  metric $h$  and any $u, v\in \Gamma(E)$  we denote
by $h(u, v)$  the smooth function  on $M$ defined  by $h(u, v) (x): = h(u(x), v(x))$ for all $x\in M$.

The paper   is  organized   as follows. In the second section
we collect  some    known  results  on geometry  of
 the space  $B^+_{i,f} (S, M)$  endowed  with the   weak  $L_2$-metric. Then we prove  one of our main technical   results (Lemma \ref{lem:normal}),  which   is needed  for our      establishing 
 the   explicit    formula  of the  Levi-Civita    connection  of the   $L_2$-metric  (Theorem \ref{thm:henrich})   as well  as for  the  proof of  our Main Theorem.   In the  last   section      first  we recall   Lee-Leung's construction  of the  almost complex structure  on  $B^+_{i,j} (S,M)$.  Then  we give   a  proof
 of the   Main Theorem.

\section{Riemannian geometry of  the space  \texorpdfstring{$B^+_{i,f}(S,M)$}{Bif(S,M)}}\label{sec:pre}

In this section,  first  we   fix   necessary notation and   recall some known  results
on the     Fr\'echet  manifold  structure  of  the  moduli   space $B^+_{i,f} (S,M)$.
Then  we  give an explicit  formula  for the  Levi-Civita  connection of the weak $L_2$-metric
on $B^+_{i,f} (S,M)$ (Theorem \ref{thm:henrich}), generalizing to an arbitrary oriented submanifold $S$ results by Henrich for  the cases  $\dim S=1$  and  
${\rm codim}\, S = 1$ \cite[Theorem 3.1, p. 25,  Theorem 5.21, p. 48]{Henrich2009} 
The proof consists in showing that Henrich's connection  $\nabla^\perp$   on the space $\Imm_f (S,M)$ \cite[Definition 5.2, p. 41]{Henrich2009} is torsion-free 
independently of the dimension of $S$. In \cite{Henrich2009} a complete proof of the torsion-freeness of $\nabla^\perp$ is only given for the case   $\dim S =1$ 
\cite[Lemma 2.17, p. 19]{Henrich2009}, but this proof is supposedly immediately generalized/adapted to  the general  case (e.g., in \cite[Definition 5.2, p. 41]{Henrich2009} the connection $\nabla^\perp$ is declared to be torsion-free without the need of a proof). Our  proof of  the torsion-freeness of $\nabla^\perp$  uses Lemma \ref{lem:normal},   which  is   possibly of independent interest and is employed  further  in the proof of the Main Theorem.

\subsection{The smooth  structure  on \texorpdfstring{$B^+_{i,f} (S,M)$}{Bif(S,M)}}\label{subs:smooth}
The   space  $B_{i,f}(S,M)$ has been  considered by Cervera-Mascar\'o-Michor in \cite{CMM1991}, see also \cite[44.2, p. 476]{KM1997} without  the assumption   that $S$ is a closed   oriented    manifold   and  in \cite{MM2005}  with the assumption that $S$  is closed.   They observed that  the  space  $\Imm_f (S,M)$   is an open invariant  subset   in   the  space $C^\infty(S,M)$ of
all smooth maps   from $S$ to $M$ endowed with  compact-open topology   and the  right action of $\Di(S)$. 
If  $S$  is  a  closed  submanifold then  any immersion $\iota: S \to M$ is proper.    In this case,  the results from the above mentioned papers  imply the  following statements (1)-(6). 

(1) $\Imm_f (S,M)$ can be naturally  endowed   with
the  structure   of an infinite   dimensional smooth manifold modeled  on   the  Fr\'echet  space  of  smooth   sections  $\Gamma (\iota ^* TM)$  along a smooth immersion
$\iota: S \to M$, see
e.g. \cite[Theorem 42.1, p. 439]{KM1997}.  

(2) $\Imm_f (S,M)$  is   the total  space  of a smooth
principal  fiber bundle  $\pi: \Imm_f (S,M) \to B_{i,f} (S,M)$  whose fiber is $\Di(S)$,   and   hence $\Imm_f (S,M)$  is also the total  space
of   a smooth   principal  fiber bundle  $\pi^+: \Imm_f (S,M)  \to B^+_{i,f} (S,M)$ whose
fiber is $\Di^+(S)$.

(3) $B_{i,f}(S,M)$ and $B^+_{i,f} (S,M)$ are Hausdorff  spaces in the quotient  topology.

(4)  For an immersion  $\iota:S\to M$ let
$N_\iota S$ be the normal bundle  $\iota^*TM/TS$,  whose fiber  at $s\in S$  can be identified with 
the  orthogonal complement  $T_{\iota(s)}\iota(S)^\perp$   of $T_{\iota (s)}\iota(S)$ in $T_{\iota(s)}M$  if  $M$ is  a  Riemannian manifold.
Then  $B^+_{i,f}(S,M)$ inherits the structure  of  a smooth Fr\'echet manifold locally modeled
on  the Fr\'echet  space $\Gamma (N_\iota S)$  of smooth   sections  of $N_\iota S$, which       is  identified
with   the kinematic tangent  space $T_{[\iota]}B^+_{i, f}(S, M)$, where   $[\iota]: =\pi^+(\iota)$. 


(5) The kinematic tangent  bundle $TB^+_{i,f}(S,M)$ has the structure  of  a smooth     vector  bundle
over  $B^+_{i,f} (S,M)$.    The  Lie bracket  of  two   kinematic  vector  fields, i.e., of two smooth sections
of $TB^+_{i,f}  (S,M)$,  is well-defined \cite[Theorem 32.8, p. 327]{KM1997}.

(6) Let $E \to  B^+_{i, f}(S,M) $ be a smooth vector  bundle, whose fiber $F$  is a  convenient
vector  space, e.g.,  $F = \R^n$. Then   we denote by   $L^k_{alt}  (TB^+_{i,f}(S,M), E)$  the  smooth bundle  over $B^+_{i,f}(S,M)$
whose fiber  over  $[\iota] \in B^+_{i,f} (S,M)$  consists  of   all  alternating
 bounded  $k$-linear maps  from   $T_{[\iota]}B^+_{i,f}(S,M) \times  \cdots  _{k\, times}\times T_{[\iota]}B^+_{i,f}(S,M)$ to $E_{[\iota]}$. We set
$$\Om^k(B^+_{i,f}(S,M), E):= \Gamma (L^k_{alt} (TB^+_{i,f}(S,M),E))$$
 and call it the space   of differential forms on $B^+_{i,f} (S,M)$  with values  in   $E$ \cite[\S 33.22, p. 352]{KM1997}.  If $E  = B^+_{i,f}(S, M) \times  \R$  is the trivial vector bundle with fiber $\R$ over  $ B^+_{i,f}(S, M) $ then
 we   abbreviate $\Om^k(B^+_{i,f}(S,M), E)$ as $\Om^k(B^+_{i,f}(S,M))$.
 
 \

From now on we    shall omit   the adjective   ``kinematic"  before  ``tangent  vector (field)".

\

 \begin{example}\label{ex:trans} (cf. \cite{LL2007,  Vizman2011})   Let   $\varphi \in   \Om^k (M)$  and   $\dim  S = r$.   Denote by $p_2: S \times \Imm_f (S,M) \to \Imm_f (S,M)$ the projection onto the second  factor.
 Since the  evaluation map 
 \begin{align*}
 ev: S \times \Imm_f (S, M) &\to M\\
  (s,\iota)&\mapsto \iota(s)
 \end{align*}
   is  smooth \cite[Corollary 3.1.3(1), p. 31]{KM1997}, for  any    $\varphi \in \Om^k(M)$   we have
 $ev^*(\varphi) \in \Om^k(S \times  \Imm_f (S,M))$. Next, using the   integration over fiber, which is a differential version of the slant map, we can push  $ev^*(\varphi)$   to a   form
 $(p_2)_* (ev)^*(\varphi) \in \Om^{k-r}(\Imm_f(S,M))$. It is not hard to see that
 $(p_2)_* (ev)^*(\varphi) = (\pi^+)^*(\varphi_B)$ where  $\varphi _B\in   \Om^{k-r}(B^+_{i,f}(S,M))$ is defined by
 \begin{equation}\label{eq:push}
 \varphi_B(X_1, \cdots, X_{k-r}):= \int _S i_{X_{1}}\cdots  i_{X_{k-r}}\varphi\bigr\vert_S,
 \end{equation}
 for  $X_1, \cdots, X_{k-r} \in \Gamma (N_\iota S)$ where $\varphi\bigr\vert_S\in \Gamma(S,\bigwedge^r\iota^*TM^*)$ is the restriction of $\varphi$ to $S$.
 If  $d\varphi =0$ then  $ d (p_2)_* (ev)^* (\varphi) = 0$    and therefore  $d\varphi_B = 0$.
 The  form $\varphi_B$ is called  the {\it transgression}   of $\varphi$.
 	\end{example}

 \begin{example}\label{ex:transv} (cf. \cite{LL2007,  Vizman2011})   Let   $\varphi \in   \Om^k (M, TM)$  and   let $\dim  S = r$. For any immersion $\iota\colon S\to M$, let  $\mathrm{vol}_{\iota^*g}$ be the volume form of $S$ associated with the induced metric $\iota^*g$, and  let $\overrightarrow {T_sS\iota}$ denote  the   unit  $r$-vector of  $T_sS$ with respect to this volum form. 
 Let $\Pi: T_{\iota(s)}M \to  N_{\iota(s)}S$ be the natural projection. Then we define the    form  $\varphi_B \in \Om^{k-r} (B^+_{i,f}S,M), TB^+_{i,f}(S,M))$
 by
 	\begin{equation}\label{eq:push2}
  \varphi_B  (X_1, \cdots, X_{k-r})(s): =\Pi\left( i_{X_1} \cdots   i_{X_{k-r}}(i_{\overrightarrow {T_sS\iota}}\varphi\bigr\vert_S)\right)
 	\end{equation}
 for  $X_1, \cdots, X_{k-r} \in \Gamma (N_\iota S)$.
 The form $\varphi_B$
  is called   the {\it transgression}  of  $\varphi$
 defined by the induced  volume form on $S$.
  \end{example}
 

\subsection{The    \texorpdfstring{$L_2$}{L2}-Riemannian metric  and its  Levi-Civita connection  on \texorpdfstring{$B^+_{i,f}(S,M)$}{Bif(S,M)}}\label{subs:aff}

In this subsection  we assume that  $(M, g)$ is a  Riemannian  manifold.
It is known  that the  space  $B^+_{i,f} (S, M)$  is endowed with   the weak  $L_2$- Riemmannian metric    
defined as follows
\begin{equation}\label{eq:l2}
\langle u,v\rangle_{[\iota]}=\int_S g(u,v) \mathrm{vol}_{\iota^*g},
\end{equation}
where $u,v\in \Gamma(N_\iota S)=T_{[\iota]}B^+_{i,f}(S,M)$. 
The RHS of (\ref{eq:l2})  extends naturally  to  a  $L_2$-metric  on  the space $\Imm_f (S,M)$, which is invariant under the  action of  $\Di^+(S)$  and therefore    it descends to   the $L_2$-metric  on $B^+_{i,f} (S,M)$  making the projection  $\pi^+\colon \Imm_f (S,M) \to B^+_{i,f}(S,M)$ a Riemannian submersion.
Kainz  proved the  existence    of the Levi-Civita
connection for the  $L_2$-metric    on the    space $\Imm_f(S,  M)$  \cite[Theorem 2.2]{Kainz1984}. Since  the  $L_2$-metric on $\Imm_f(S, M)$ is  $\Di^+ (S)$-invariant,  the associated  Levi-Civita   connection is  also  $\Di^+(S)$-invariant. A  formula  for the  Levi-Civita
connection  on $B^+_{i,f}(S,M)$  has been given by Michor and Mumford via the   equation for geodesics on $B^+_{i,f}(S,M)$ \cite[\S 4.2]{MM2005}   and explicitly  by
  Henrich in his PhD Thesis   for the
case  $\dim S =1$ or $\codim\, S = 1$ \cite[Theorem 3.1, Theorem 5.21]{Henrich2009}.   In  the remaining part of this subsection
we  shall   generalize  Henrich's formula for an arbitrary  closed  oriented   manifold $S$.

  First we recall  the notion of  an {\it affine connection} and its associated  {\it   covariant  differentiation}  on   a  smooth  possibly infinite dimensional  manifold  $\Mf$  modelled
  on a convenient vector  space, following \cite{KM1997}, \cite{MMM2013}, and \cite{Kainz1984}.
Denote by $\X(\Mf)$ the space  of smooth vector  fields on $\Mf$.
\begin{definition}
	\label{def:aff}   An affine connection  on   a   smooth  manifold $\Mf$  is a smooth mapping
	$$\nabla:\X (\Mf) \times  \X(\Mf) \to \X(\Mf)$$
which is   a derivation   on the  second   argument  and $C^\infty(\Mf)$-linear   on the 
first  argument. In other words,   re-denoting   $ \nabla (X, Y)$ as $\nabla _X (Y)$, we have
\begin{eqnarray}
\nabla _X (Y+Z) = \nabla _X(Y)  +  \nabla _X(Z)\\
 \nabla _X (fY) = f \nabla _X(Y)  +  (Xf) Y \label{eq:leibniz1}\\
\nabla _{f X + g Y}  (Z) = f\nabla_{X} Y +  g \nabla_{Y} Z \label{eq:lin1}
\end{eqnarray}
for all  $ X, Y, Z \in \X (\Mf)$ and $  f, g \in C^\infty (\Mf)$.
The operator $\nabla _X$ is  called  a {\it  covariant  differentiation } w.r.t. $X$. If the manifold $\Mf$ is equipped with a Riemannian metric
\[
\langle\,,\,\rangle\colon   \X (\Mf) \times \X (\Mf) \to C^\infty(\Mf),
\]
then the covariant  derivation is called  {\it Levi-Civita}, if   it respects  the metric and it is torsion-free, i.e.,
\begin{eqnarray}
X \la Y,  Z\ra  = \la \nabla_X  Y, Z \ra + \la Y, \nabla _X Z \ra\label{eq:LC1}\\
T^\nabla (X, Y) : = \nabla_X Y - \nabla _Y X - [X, Y] =0. \label{eq:LC2}
\end{eqnarray} 
As in the finite dimensional case, the above conditions determine $\nabla$ uniquely, so that, when it exists we talk of  {\it the Levi-Civita connection}.
\end{definition}

\begin{remark}\label{rem:LV}
	   It is not hard to   see that the above definition of  an  affine connection on $\Mf$ is equivalent to  the definition  of an  affine connection   via a   smooth projection $K: T^2 \Mf \to T\Mf$ \cite[Definition 37.2, p. 376]{KM1997}. 

Furthermore, the   notions   of an affine   connection and of the associated covariant dfferentiation on $\Mf$ are particular cases   of  those  of a  linear connection  on a smooth vector  bundle  $E$  over $\Mf$ and of its associated covariant   differentiation  $\nabla : \X(\Mf)\times  \Gamma(E) \to \Gamma(E)$      \cite[37.27, 37.28, p. 397]{KM1997}.
\end{remark}

\begin{remark}\label{rem:nabla-on-paths}
As $(X,Y)\mapsto \nabla_XY$ is $C^\infty(\Mf)$-linear in the first variable, the datum of a covariant derivative $\nabla$ is equivalent to the datum of all the derivatives 
\[
\nabla_u\colon \X(\Mf) \to T_p\Mf
\]
 with $u\in T_p\Mf$ ranging in the tangent  space of $\Mf$ at $p$, with $p$ varying over the whole of $\Mf$.
\end{remark}

\begin{remark}
The torsion of an affine connection is a tensor, i.e., the value $T^\nabla(X,Y)_p$ of the vector field $T^\nabla(X,Y)$ at the point $p$ of $\Mf$ only depends on the values $X_p$ and $Y_p$ of $X$ and $Y$ at $p$. In other words, one has a well defined bilinear map $u\otimes v\mapsto T^\nabla_p(u,v)$ defined by
\[
T^\nabla_p(u,v)=T^\nabla(X,Y)_p
\]
where $X$ and $Y$ are arbitrary extensions of the tangent vectors $u$ and $v$ at the point $p$, respectively, to tangent vector fields on a 2-dimensional surface in $\Mf$.
\end{remark}

Next, following  \cite{MM2005}, cf.\cite{Henrich2009},  we  shall      express  the Levi-Civita connection   $\nabla^{LC}$ on $B^+_{i,f}(S, M)$ w.r.t.  the  $L_2$-metric   as  a  sum  of a  torsion-free affine connection $\nabla^\perp $ on $B^+_{i,f}(S, M)$   and a  symmetric tensor  $\B$. 
  By Remark  \ref{rem:nabla-on-paths},  to define $\nabla ^\perp$   it   suffices  to  define
$\nabla_{u}^\perp$ for every tangent vector $u$ in $T_{[\iota]}B^+_{i,f}(S, M)$. Let $\gamma\colon \mathbb{R}\to \Imm_f(S,M)$ be a path of  free immersions with $\gamma(0)=\iota$ whose velocity vector $\dot{\gamma}\vert_{t=0}$ at $t=0$ is a lift/a representative of the tangent vector $u$  in  $\Gamma(\iota^*TM)$.  For all $k \in \N$ we have isomorphisms
\begin{equation}\label{eq:cartesian}
C^\infty (\R^{k} , C^\infty (S,M)) = C^\infty  (\R^{k} \times  S, M).
\end{equation}
By means of this for $k=1$,
we identify     $\gamma$  with a map, which we will denote by the same symbol,  $\gamma: \R \times S \to  M$,
\cite[Corollary 3.13, p. 31]{KM1997}.
At any point $s$ of $S$, $\gamma(\cdot, s)$ defines a curve in $M$ whose velocity vector $u (s)$ at $t=0$ is a vector in $T_{\gamma(0,s)} M$.  
 For any vector field $X$ on $B^+_{i,f}(S, M)$,    and   for any fixed $t$, the vector $X_{[\gamma(t)]}$ can be identified with a section of the normal bundle $N_{\gamma(t)}S$. 
By identifying  $N_{\gamma(t)}S$ with a subbundle of the tangent bundle $(\gamma(t))^* TM$, e.g., by means of the Riemannian metric as in  Subsection \ref{subs:smooth} (4),
 we can think of $X_{[\gamma]}:=X\circ\gamma$ as a section of $\gamma^*TM$. The pullback of the Levi-Civita connection on $TM$ gives a connection $\gamma^*(\nabla^{LC;M})$ on $\gamma^*TM$ and we can define
\begin{equation}\label{eq:perp1}
(\nabla_{u}^\perp X):=(\gamma^*(\nabla^{LC;M})_{\frac{\partial}{\partial t}} X_{[\gamma]}\bigr\vert_{t=0})^\perp, 
\end{equation}
where on the right hand side $(-)^\perp$ denotes its projection on the normal bundle $N_{[\iota]}S$. 

 \begin{proposition} \label{prop:pull1}(cf. \cite[Lemma 2.17]{Henrich2009})  The  formula  (\ref{eq:perp1})  defines  a    torsion-free  connection  on   $B^+_{i,f}(S, M)$.
 \end{proposition}

\begin{proof} One immediately sees that (\ref{eq:perp1}) defines an affine connection, so we are only left with showing that 
the torsion $T^\perp(u,v)$ of $\nabla ^\perp$ vanishes  for any $u, v \in  T_{[\iota]}B^+_{i,f}(S,M)= N_{\iota}S$. To do this 
we  shall use the following  argument, which shall be utilized
several times  in our  paper.

 We consider a deformation 
 $$\gamma:  (-\eps, \eps) \times (-\eps, \eps) \times S \to M, \; (x,y,s)\mapsto \gamma (x,y,s),$$
    such that  for all $s\in S$, 
    $$\gamma (0, 0, s)=  \iota (s),\quad  {\frac{d\gamma}{dx}}\biggr\vert_{(0,0,s)}= u(s), \quad\text{  and }\quad {\frac{d \gamma}{dy}}\biggr\vert_{(0,0,s)}= v(s).$$
    
 More  explicitly we   set
 \begin{equation}\label{eq:Exp}
 \gamma (x, y, s): = \Exp_{\iota(s)}\,  (x\cdot u(s)+ y\cdot v(s)),
 \end{equation}
 where $\Exp_{\iota(s)}:  T_{\iota(s)}M \to M$ is the exponential map for the Riemannian manifold $M$ at the point $\iota(s)$.
This deformation extends the    vectors  $u, v \in T_{[\iota]}B^+_{i,f}(S,M)=N_\iota S$  to   vector  fields,  
	that we denote by $X$ and $Y$, along 
	the image  of ${\gamma} ((-\eps, \eps) \times (-\eps, \eps))$ in  $\Imm_{f} (S,M)$   via  the isomorphism (\ref{eq:cartesian}),
i.e., the  vector field $X$ is defined by 
	 \begin{equation}\label{eq:ExpX}
	X_{\gamma(x,y)}(s)=(d\Exp_{\iota(s)})\vert_{x\,u(s)+y\, v(s)}(u(s))
	\end{equation}
	and similarly for $Y$. 		
	The projection $\pi^+: \Imm_f (S, M) \to B^+_{i,f} (S,M)$ maps  $X,Y$ to vector fields $X^{\mathrm{ver}}, Y^{\mathrm{ver}}$ along  the  image of $\pi^+ \circ  {\gamma}((-\eps, \eps) \times (-\eps, \eps))$.
	By definition  we   can identify  $X^{\mathrm{ver}}(\pi^+ \circ {\gamma}(x,y))(s)$  with  the   projection
	of $X(s)$  on the  normal  bundle $N_{{\gamma}(x,y,s)}S$, i.e.,  we have  the following decomposition
	\begin{equation}\label{eq:ver}
	X({\gamma}(x,y)) = X^{\mathrm{ver}}(\pi^+\circ{\gamma}(x,y))  +  X ^{\mathrm{hor}}(\pi^+\circ{\gamma}(x,y)),
	\end{equation}
	where $X^{\mathrm{hor}}(\pi^+\circ{\gamma}(x,y)) \in   \X({\gamma}(x,y)S)$. By construction, the vector fields $X^{\mathrm{ver}}, Y^{\mathrm{ver}}$ extend the tangent vectors $u,v$, so the torsion $T^\perp(u,v)$ is computed by $T^\perp(u,v)=\left(\nabla^\perp_{X^{\mathrm{ver}}} Y^{\mathrm{ver}} - \nabla^\perp _{Y^{\mathrm{ver}}} X^{\mathrm{ver}} - [X^{\mathrm{ver}}, Y^{\mathrm{ver}}]\right)_{[\iota]}$. Since  $X^{\mathrm{ver}} = (\pi^+)_* X$  and $[X, Y]= 0$, we have  $[X^{\mathrm{ver}}, Y^{\mathrm{ver}}] = 0$. The conclusion of the proof is then immediate from the following Lemma  \ref{lem:normal}.
	\end{proof}
	
	\begin{lemma}\label{lem:normal}  Assume that  $X^{\mathrm{ver}}, Y^{\mathrm{ver}}$ are generated by $\gamma$ defined  in  (\ref{eq:Exp})  and (\ref{eq:ver}). Then  we have 
		\begin{equation}
		\label{eq:vanish1}
		\nabla_{X^{\mathrm{ver}}}^\perp {Y^{\mathrm{ver}}}_{[\iota]} =\nabla_{Y^{\mathrm{ver}}}^\perp {X^{\mathrm{ver}}}_{[\iota]}=0.
		\end{equation} 
	\end{lemma}
	\begin{proof}  
	As the statement is symmetric in $X^{\mathrm{ver}}$ and $Y^{\mathrm{ver}}$, we only need to show that $\nabla_{X^{\mathrm{ver}}}^\perp {Y^{\mathrm{ver}}}_{[\iota]} =0$.
	To prove   Lemma \ref{lem:normal}, using (\ref{eq:perp1}), it suffices
		to show  that for  any  $w \in \Gamma (N_\iota S)$  and  any $s\in S$ we  have  
		\begin{equation}\label{eq:vanish2}
		g (\gamma^*(\nabla^{LC;M})_{\frac{\partial}{\partial x}} {Y^{\mathrm{ver}}}_{[\gamma]}{|_{(x,y)= (0,0)}}, w)_{s} = 0.
		\end{equation}
		Since $\gamma$ is a restriction of the  exponential map $\Exp:  TM \to M$    at  $\iota(S)$, we have
		\begin{equation}\label{eq:vanish3}
		\gamma^*(\nabla^{LC;M})_{\frac{\partial}{\partial x}} Y_{[\gamma]}|_{(x,y)= (0,0)}= 0.
		\end{equation}
		Using \ref{eq:vanish3}, to  prove  (\ref{eq:vanish2})  it suffices to show  that
		
		\begin{equation}\label{eq:vanish4}
		g( \gamma^*(\nabla^{LC;M})_{\frac{\partial}{\partial x}} Y^{\mathrm{hor}}_{[\gamma]}{|_{(x,y)= (0,0)}}, w)_s = 0.
		\end{equation}
	Abusing  the same notation $\gamma$,  now we   let  $\gamma$ be a map     from  $(-\eps, \eps)\times (-\eps,\eps )\times (-\eps, \eps) \times  S  \to   M$  defined  by (cf. (\ref{eq:Exp}))
	\begin{equation}\label{eq:Exp3}
	\gamma(x,y,z, s): =\Exp_{\iota(s)}(x\cdot u(s)+ y \cdot v(s) + z \cdot w(s) ),
	\end{equation}
	and let $Z$ be the vector field along the image of $\gamma$ defined in analogy to $X$ and $Y$.
		Then,  as in (\ref{eq:vanish3}),  we have
	 \begin{equation}\label{eq:vanish6}
	\gamma^*(\nabla^{LC;M})_{\frac{\partial}{\partial x}} Z_{[\gamma]}|_{(x,y,z)= (0,0,0)} =0.
	\end{equation}
and so equation (\ref{eq:vanish4}) is equivalent to
	\begin{equation}\label{eq:vanish4.5}
		\gamma^*(\nabla^{LC;M})_{\frac{\partial}{\partial x}}g ( Y^{\mathrm{hor}}_{[\gamma]}, Z_{[\gamma]})_{|(x,y,z)=(0,0,0)}=0.
		\end{equation}
Noting that $g(Y^{\mathrm{hor}}, Z^{\mathrm{ver}})=0$, to  prove  (\ref{eq:vanish4})	  it suffices  to  show  that
		\begin{equation}\label{eq:vanish5}
		\gamma^*(\nabla^{LC;M})_{\frac{\partial}{\partial x}}g ( Y^{\mathrm{hor}}_{[\gamma]}, Z^{\mathrm{hor}}_{[\gamma]})_{|(x,y,z)=(0,0,0)}=0.
		\end{equation}
		Since  $ Y\circ \gamma: (-\epsilon,\epsilon)^3\times  S \to  TM$  is a differentiable map,  the  composition  $ Y^{\mathrm{hor}} \circ \gamma : (-\eps, \eps)^3 \times S \to TM$ is a differentiable  map. Since  
		$Y^{\mathrm{hor}}\circ \gamma (0,0, 0, s) = 0$, we   can  write  for $x \in [-\eps/2, \eps/2]$  
		\begin{equation}\label{eq:taylor1}
		Y^{\mathrm{hor}}\circ\gamma(x,0,0, s)= x \cdot Y'(x, 0,0, s)  
		\end{equation}
		where $Y'(x, 0,0, s) \in  T_{\gamma (x, 0,0, s)}M$  and  $|Y ' (x, 0, 0, s)| \le   A(s) $  for some   positive  number  $A(s) \in \R$.  Since   $S$ is compact,   we can choose  $A(s)$
		independent     of $s$. The same argument applies to $Z\circ\gamma$.
		This   yields  (\ref{eq:vanish5})  immediately  and completes
		the proof of Lemma \ref{lem:normal}
	\end{proof}

Next we shall  define  the desired   symmetric tensor  $\B$. As $(M,g)$ is a Riemannian manifold and $S$ is 
oriented, we have a volume form map

\begin{align*}
\mathrm{vol}_S\colon \mathrm{\Imm}_{i,f}(S,M)&\to \Omega^{\dim S}(S;\mathbb{R})\\
\iota&\mapsto \mathrm{vol}_{\iota^*g},
\end{align*}
where $\mathrm{vol}_{\iota^*g}\in \Omega^{\dim S}(S;\mathbb{R})$ denotes the volume form on $S$ associated with the pullback metric $\iota^*g$. As $\mathrm{vol}_S$ is a smooth function on $\mathrm{\Imm}_{i,f}(S,M)$ with values in a vector space, we can take its derivative $X\mathrm{vol}_S$ with respect to a smooth vector field on $\mathrm{\Imm}_{i,f}(S,M)$ and this will again be a smooth $\Omega^{\dim S}(S;\mathbb{R})$-valued function on $\mathrm{\Imm}_{i,f}(S,M)$.
If $\tilde{\iota}=\iota\circ\phi$, then $\phi\colon (S,\iota^*g)\to (S,\tilde{\iota}^*g)$ is an isometry. This implies that the equation
\begin{equation}\label{eq:gradient}
g(X_{[\iota]}, W_{[\iota]}) (\mathrm{vol}_S)_{[\iota]} =(X\mathrm{vol}_S)_{[\iota]} 
\end{equation}
for any vector field $X$ on  $B^+_{i,f}(S,M)$, is well-defined.   Namely, the equivariance conditions $(\mathrm{vol}_S)_{\tilde{\iota}}=\phi^*((\mathrm{vol}_S)_{\iota})$ and $X_{\tilde{\iota}}=\phi^*X_{\iota}$ imply that  a   solution  $W$ is $\mathrm{Diff}^+(S)$-equivariant as well.  

By the very definition of the mean curvature vector field for an immersed submanifold we have the following.

\begin{lemma}\label{lem:MCV}  The equation  \ref{eq:gradient}  has  a unique solution  $W_{[\iota]}= -   (\dim S) H_{\iota(S)}$, where
	$H_{\iota(S)}$ is the mean curvature   of the immersed  submanifold  $\iota(S)$, defined by
	\[
	H_{\iota(S)}=\frac{1}{\dim S} \sum_{i=1}^{\dim S}   (\nabla^{LC;M}_{e_i}     e_i )^\perp,
	\]
	where the $e_i$ are orthonornal tangent vector fields on a neighborhhod of a point $s$ in $S$ and $\perp$ denotes the projection on the normal bundle.
	 The solution  $W_{[\iota]}=-(\dim S)H_{\iota(S)}$ defines a vector field $W$ on $B^+_{i,f}(S,M)$.
\end{lemma}


Consider now the bilinear form 
\[
\B:  T B^+_{i,f} (S, M) \times   TB^+_{i,f} (S, M)   \to TB^+_{i,f} (S,M)
\]
defined by
\begin{equation}
\B(u, v)_{[\iota]} =  g(  u, v) W_{[\iota]} - g( u,W_{[\iota]})  v  - g ( v, W_{[\iota]})  u  \label{eq:b}
\end{equation}
for any $u,v\in T_{[\iota]}B^+_{i,f}(S,M)$,where $W$ is the multiple of the mean curvature vector field defined in Lemma \ref{lem:MCV}, and where we use the $C^\infty(S;\mathbb{R})$-module structure on $T_{[\iota]} B^+_{i,f} (S, M)=N_\iota S$.

The following theorem   generalizes to an arbitrary $S$ results by Henrich for  the cases  $\dim S=1$  and  
${\rm codim}\, S = 1$ \cite[Theorem 3.1, p. 25,  Theorem 5.21, p. 50]{Henrich2009}.
\begin{theorem}\label{thm:henrich}  The      covariant  derivation $\nabla ^\perp  - {\frac{1}{2}}  \B $  is  the Levi-Civita    covariant    derivation on $B^+_{i,f}(S, M)$.
\end{theorem}

\begin{proof} By uniqueness, we only need to show that  $\nabla ^\perp  - {\frac{1}{2}}  \B$ is torsion-free and compatible with the Riemannian metric on $B^+_{i,f}(S, M)$. The torsion-freeness is immediate, as $\nabla ^\perp$ is torsion-free, and  $\B$ is    a symmetric bilinear  form.   Next, we show that  
	\[
	X \la Y,  Z\ra   = \la \nabla^\perp_X  Y, Z \ra + \la Y, \nabla^\perp _X Z \ra  + {-\frac{1}{2}}(\la \B(X,Y), Z \ra + \la Y, \B(X,Z) \ra)
	\]
	holds  for any $X, Y, Z \in \X (B^+_{i,f}(S, M))$. As the difference between the left and the right hand side of the above equation is a tensor, without loss of generality we can assume that  $X, Y, Z $ are generated by 
	 a three-parameter  variation 
	$\bar \gamma: \R^3 \to B^+_{i,f} (S, M)$  with  
	$\bar \gamma (0,0,0) = [\iota]$,  e.g., as in (\ref{eq:Exp3}).  The map $\bar \gamma$  lifts  to
	a smooth  map $ \gamma: (-\epsilon,\epsilon)^3 \to \Imm_f(S, M)$    defined  by   a smooth map
	\begin{align*}
	\gamma : (-\epsilon,\epsilon) \times (-\epsilon,\epsilon) \times (-\epsilon,\epsilon) \times S &\to M,
	\\  
	(x, y, x, s) &\mapsto \gamma (x, y, z, s),\\
	\end{align*}  
with $\gamma (0, 0, 0, s) = \iota (s)$	and such  that $\gamma$  is purely normal at $(x,y,z)=(0,0,0)$. 
Then we can assume  at every $(x,y,z)$ we have
	$$ X = \left(\frac{d \gamma }{dx}\right)^{\mathrm{ver}}, \:   Y = \left(\frac{d \gamma}{dy}\right)^{\mathrm{ver}}, \:  Z = \left(\frac{d \gamma }{dx}\right)^{\mathrm{ver}}, $$
and find	
	\begin{align*}\label{eq:metric}  
	\left(X \la Y, Z \ra\right)_{[\iota]} &= \frac{d}{dx}  \int_S g(Y_{[\gamma]},Z_{[\gamma]})  (\mathrm{vol}_S)_\gamma\biggr\vert_{(x,y,z)=(0,0,0)}
	\\
	&=
	\int_S \left( \frac{d}{dx}  g(Y_{[\gamma]},Z_{[\gamma]})  \right) (\vol_S)_\gamma\biggr\vert_{(x,y,z)=(0,0,0)}
	+   \int_S g(Y_{[\gamma]},Z_{[\gamma]}) \left(\frac{d}{dx} (\mathrm{vol}_S)_\gamma\right)\biggr\vert_{(x,y,z)=(0,0,0)}\\
	&=
	\int_S \left(  g( \gamma^*(\nabla^{LC;M})_{\frac{\partial}{\partial x}} Y_{[\gamma]}),Z_{[\gamma]}) +g( Y_{[\gamma]}, \gamma^*(\nabla^{LC;M})_{\frac{\partial}{\partial x}} Z_{[\gamma]}))   \right) (\mathrm{vol}_S)_\gamma\biggr\vert_{(x,y,z)=(0,0,0)}\\
	&\qquad\qquad
	+   \int_S g(Y_{[\gamma]},Z_{[\gamma]}) \left(\frac{d}{dx} (\mathrm{vol}_S)_\gamma\right)\biggr\vert_{(x,y,z)=(0,0,0)}\\
	&=
	\int_S \left(  g( (\nabla^\perp_X Y)_{[\iota]},Z_{[\iota]}) +g( Y_{[\iota]}, (\nabla^\perp_XZ)_{[\iota]}   \right) (\vol_S)_{[\iota]}
	+   \int_S g(Y_{[\iota]},Z_{[\iota]}) (X\mathrm{vol}_S)_{[\iota]}\\
	&=\langle \nabla^\perp_X Y,Z\rangle_{[\iota]} + \langle  Y, \nabla^\perp_XZ\rangle_{[\iota]} +\int_S g( X_{[\iota]}, W_{[\iota]}) g(Y_{[\iota]},Z_{[\iota]})  (\mathrm{vol}_S)_{[\iota]}.
	\end{align*}
	In the above equation we used that, as $Z_{[\iota]}=Z_{[\iota]}^{\mathrm{ver}}$, we have
	\begin{align*}
	g( \gamma^*(\nabla^{LC;M})_{\frac{\partial}{\partial x}} Y_{[\gamma]}),Z_{[\gamma]})\biggr\vert_{(x,y,z)=(0,0,0)}&=g( (\gamma^*(\nabla^{LC;M})_{\frac{\partial}{\partial x}} Y_{[\gamma]}))^\perp,Z_{[\gamma]})\biggr\vert_{(x,y,z)=(0,0,0)}\\
	&=g( (\nabla^\perp_X Y)_{[\iota]},Z_{[\iota]}).
	\end{align*}
	Therefore, we are reduced to showing that
	\[
	\int_S g(X_{[\iota]},W_{[\iota]})  g(Y_{[\iota]},Z_{[\iota]})(\mathrm{vol}_S)_{[\iota]} =-\frac{\la \B(X,Y), Z \ra + \la Y, \B(X,Z) \ra}{2},
	\]
which is straightforward from the definition of $\B$.
\end{proof}

\begin{remark}\label{rem:LC}

 In \cite[(8)]{MMM2013}, see also \cite[\S 4.2]{Michor2016},   Micheli, Michor, and Mumford give more generally an existence condition for  the Levi-Civita   covariant  derivative   on  a  smooth manifold $M$  modeled on convenient locally convex vector
	spaces and 
	  endowed with a weak Riemannian  metric $g$, i.e.  a symmetric positive definite bilinear form $g$ such that $g^\flat_x : T_xM \to T^\ast_xM$
	  is only injective for each $x\in M$,  in terms of  symmetric gradients with respect to $g$.
\end{remark}

\section{The  formally  K\"ahler  structure of higher   dimensional  knot spaces}\label{sec:formal}

In this  section we give a  proof of our  Main Theorem \ref{thm:formalK}. First we  recall  Lee-Leung's construction
of  the almost complex structure  $J$ on $B^+_{i,f}(S,M)$  associated with a VCP on  $M$. Then  we  derive  Theorem \ref{thm:formalK}  from
Corollary \ref{cor:tor1} and  Proposition \ref{prop:almost-complex-connection}.

  Denote
by  $Gr^+(r,M)$  the    Grassmanian  bundle   over $(M,g)$ whose fiber over $p\in  M $  consists  of  the Grassmanian manifold of oriented
$r$-dimensional subspaces   in $T_pM$.  For  $ v\in Gr^+_p(r,M)$ let us denote by $\overrightarrow{v}$  the oriented unit $r$-vector associated   to $v$ and to the Riemannian metric $g$ of $M$, and by $v ^\perp$ the   oriented    orthogonal  complement  to $v$ in $T_pM$.
 Lee-Leung made  the following  simple  but  crucial observation  \cite[p. 146]{LL2007}
 \begin{lemma}\label{lem:jcvp} Let  $\chi \in \Om^{r+1}(M, TM)$ be a VCP. Then for  any $v \in Gr^+_p(r, M)$  and  any  $\xi \in v^\perp\subset T_pM$ we have
 $$i_{\overrightarrow v}\chi_p (\xi) \in v^\perp.$$
 Furthermore, the restriction    of $i_{\overrightarrow v}\chi_p$ to $v^\perp$ is a  1-fold VCP on $v^\perp$,  denoted   by $J(\chi, v)$,  which satisfies  the following  relation    for any $\xi, \zeta \in v^\perp$
 \begin{equation}\label{eq:sympl}
 i_{\overrightarrow v}(\varphi_{\chi})_p (\xi, \zeta) = g_p (J(\chi,v)\xi, \zeta),
 \end{equation}
 where $\varphi_\chi$ is the VPC-form of $\chi$ (equation (\ref{eq:varphi-chi})). 
 \end{lemma}
We observe that  Lemma \ref{lem:jcvp}  is a consequence  of  Gray's theorem \cite[Theorem 2.6]{Gray1969}.
\begin{remark}
As 1-VCP on Euclidean vector spaces are equivalent to linear K\"ahler structures, one is naturally led to considering the
Hermitian  vector  bundle $E(\chi)$   over  $Gr^+(r,M)$  whose fiber  over $v$  is $v^\perp$  with the 
Hermitian   structure  $(g, J(\chi, v))$.  As Chern classes are deformation invariants of the complex structure on a real vector bundle, the Chern classes   of $E(\chi)$  are deformation  invariants  of the 
VCP  structure $\chi$ on $M$. A detailed investigation of these and other deformation invariants of a VCP  structure with emphasis on $G_2$- and  torsion-free  Spin(7)-manifolds, as well as a comparison with known  deformation  invariants  of    $G_2$-structures \cite{CGN2018} will appear elsewhere.
\end{remark}

\

Assume the dimension of the closed oriented manifold $S$ is $\dim_{\mathbb{R}}S=r$. 
At each point $(s,\iota)$ in $S\times \Imm_f (S, M)$, the image $\iota_*T_sS$ of the tangent space $T_sS$ in $T_{\iota(s)}M$ defines an element in $Gr^+(r,M)$, so we have defined a smooth map
\[
v\colon S\times \Imm_f (S, M) \to Gr^+(r,M).
\]
By construction, $v^\perp(s,\iota)=N_{\iota(s)}S$, so by Lemma \ref{lem:jcvp} we have a pointwise 1-VCP on the fibres of the normal bundle to $S$ at $\iota$, compatible with the inner product induced from $M$. From this it follows that  we have the linear operator $J$ on the tangent bundle $TB^+_{i,f}(S,M)$ defined by setting, at each point $[\iota\colon S\to M]$,
\begin{equation}\label{eq:J}
(J_{[\iota]}X)_s: = i_{\overrightarrow{T_{\iota(s)}S}}\chi_{\iota(s)} (X_s).
\end{equation}
Lee and Leung observed that and  $J$     
is  an almost  complex structure  on $TB(S, M)$ \cite[p. 146]{LL2007}, compatible with the $L_2$-metric, thus endowing $B^+_{i,f}(S,M)$ with the structure  of  an almost Hermitian  manifold. 
As the fundamental 2-form of this Hermitian structure  is the   transgression  defined by  (\ref{eq:push2})   of  the   $(r+2)$-VCP-form $\varphi_\chi$  defined in (\ref{eq:varphi-chi}), \cite[Lemma 7]{LL2007}, one sees that 
if  $\varphi_\chi$ is closed, then   $B^+_{i,f}(S,M)$ is an almost K\"ahler manifold.

\medskip

In the remaining  part  of this section we  assume that $\chi$ is a parallel   VCP on
$(M,g)$. In particular, this implies that $\varphi_\chi$ is closed.

%
%

\begin{proposition}\label{prop:almost-complex-connectionf}
The almost complex structure $J$ is parallel with respect to the affine connection $\nabla^\perp$ of equation (\ref{eq:perp1}), i.e. $\nabla^\perp J=0$. Equivalently, $\nabla^\perp$ is an almost complex connection with respect to the almost complex structure $J$.
\end{proposition}
\begin{proof}
We need to prove that $\nabla^\perp_ XJ=0$ for any vector field $ X$ on $B^+_{i,f}(S,M)$.    As $(\nabla^\perp_ XJ) Y=\nabla^\perp_ X(J Y )-J(\nabla^\perp Y)$, this is equivalent to proving 
\begin{equation}\label{eq:left-side-is-a-tensor}
\langle \nabla^\perp_ X(J Y)-J(\nabla^\perp_X  Y), Z\rangle_{[\iota]}=0
\end{equation}
for any three vector fields $X, Y, Z$ on $B^+_{i,f}(S,M)$, and any point $[\iota] \in B^+_{i,f}(S,M)$. 
 Since the left hand side of (\ref{eq:left-side-is-a-tensor}) is a tensor,
without loss of generality, as in the   proof  of Theorem  \ref{thm:henrich},  we assume that $X,  Y,  Z$  are    the vertical vector fields generated
by the three-parameter   variation
\begin{align*}\gamma: (-\epsilon,\epsilon)\times (-\epsilon,\epsilon)\times(-\epsilon,\epsilon) \times  S &\to M\\
 (x, y, z, s) &\mapsto  \gamma(x, y, z, s),\\
 & \gamma(0,0,0,s) =\iota(s)
 \end{align*}
 defined by  (\ref{eq:Exp3}) and therefore  we can apply  Lemma \ref{lem:normal}, giving $\nabla^\perp _XY=0$.  From the defining equation (\ref{eq:Exp3}), the restrictions of the vertical vector fields $X,Y$ and $Z$ at $(x,y,z)=(0,0,0)$ are the sections $u,v$ and $w$ of $N_\iota S$, respectively.
We are then reduced to proving that 
$$\langle \nabla^\perp_X(JY),Z\rangle_{[\iota]}=0,$$
i.e.,  due to fact that by (\ref{eq:perp1})  the value $\nabla ^\perp _X  (JY)_{\iota(s)}$ depends
only  on $u(s)$ and due to the arbitrariness of  the restriction $w$ of $Z$ at $[\iota]$, to show that
\begin{equation}\label{eq:vanish8}
g(\nabla ^\perp _u (JY), w)_{\iota (s)}  = 0
\end{equation}
 for  $Y$  the vertical vector field  defined by (\ref{eq:Exp3}), for any $u,w\in \Gamma(N_\iota S)$ and  for all $s \in  S$, which we prove next.

From the definition of $\nabla^\perp$, equation (\ref{eq:perp1}), and the definition of $J$, equation (\ref{eq:J}),  we get

\begin{align}\label{eq:vcp2}
 g(\nabla_u^\perp(J Y),w)_{\iota(s)}&=
g(\bigl(\gamma ^*(\nabla ^{LC;M})_{{\frac{\p}{\p x}}} \chi(\overrightarrow{T_{\gamma(x, s)} \gamma(x) (S)},{Y}) \bigr)^{\perp}, w)_{\iota(s)}.
\end{align}
As $\chi$ is parallel with respect to the Levi-Civita connection on $M$, and since $w \in \Gamma( N_\iota S)$, we get
\begin{align*}
g(\nabla^\perp _u(JY),& w)_{\iota(s)} = g( \gamma^* (\nabla ^{LC; M})_{ {\frac{\p}{\p x}}} \chi (\overrightarrow{T_{[\gamma]}S}, {Y}), w)_{\iota(s)}  \\
\notag& =g( \chi (\gamma^*(\nabla^{LC;M})_{ {\frac{\p}{\p x}}}\overrightarrow{T_{[\gamma]}S}, v), w)_{\iota(s)}   + g(J\nabla^\perp_X Y, w)_{\iota(s)}\\
\notag&= \varphi_\chi (\gamma^*(\nabla^{LC;M})_{ {\frac{\p}{\p x}}}\overrightarrow{T_{[\gamma]}S}, v, w)_{\iota(s)}, 
\end{align*}
where we used $\nabla^\perp_XY=0$ from Lemma \ref{lem:normal}.

Let ${E}_1(x,s), \cdots, {E}_r(x,s)$ be the  basis  of   $T_{\gamma (x, s)}  \gamma (x)S$ obtained by the exponential flow in the direction $X$ from an orthonormal basis $e_1(s), \cdots, e_r(s)$ of $T_{\iota(s)}\iota(S)\subset T_{\iota(s)}M$, so that
$$  \overrightarrow{T_{\gamma(x, s)} \gamma(x) (S)}  =     \vol  ({E}_1 (x,s) \wedge \cdots  \wedge  {E}_r(x,s))^{-1}  {E}_1(x,s) \wedge   \cdots   {E}_r (x, s).$$

Then  we have
\begin{align}\label{eq:calib}
\varphi_\chi & (\gamma^*(\nabla^{LC;M})_{ {\frac{\p}{\p x}}}\overrightarrow{T_{[\gamma]}S}, v, w)|_{x=0} \\
&= \gamma^*(\nabla^{LC;M})_{\frac{\p}{\p x}} ( (\vol  ({E}_1 (x,s)\wedge \cdots \wedge  {E}_r(x,s))^{-1}) \bigr\vert_{x=0} \varphi_\chi(e_1,\dots e_r, v, w)\nonumber\\
&\notag +\sum_{i=1}^r \varphi (e_1, \cdots  , \gamma^*(\nabla^{LC;M})_{ {\frac{\p}{\p x}}} {E}_i \bigr\vert_{x=0}, \cdots, e_r,  v, w).
\end{align}
From the defining equation (\ref{eq:gradient}) for the gradient vector field $W$ we have
\begin{align*}
\gamma^*(\nabla^{LC;M})_{ {\frac{\p}{\p x}}}&( (\vol  ({E}_1 (x,s)\wedge \cdots \wedge  {E}_r(x,s))^{-1})\bigr\vert_{x=0}\\
&=-g(u,W_{[\iota]}),
\end{align*}
so equation (\ref{eq:calib}) becomes 
\begin{align}\label{eq:calib2}
\varphi_\chi & (\gamma^*(\nabla^{LC;M})_{ {\frac{\p}{\p x}}}\overrightarrow{T_{[\gamma]}S}\bigr\vert_{x=0}, v, w) \\
&= -g(u,W_{[\iota]}) \varphi_\chi(e_1,\dots e_r, v, w)\nonumber\\
&\notag +\sum_{i=1}^r \varphi_\chi (e_1, \cdots  , \gamma^*(\nabla^{LC;M})_{ {\frac{\p}{\p x}}} {E}_i\bigr\vert_{x=0}, \cdots, e_r,  v, w)
\end{align}

Denote by $(\gamma^*(\nabla^{LC;M})_{ {\frac{\p}{\p x}}} E_i)|_{x=0}^T$ the projection of
$(\gamma^*(\nabla^{LC;M})_{ {\frac{\p}{\p x}}} E_i)|_{x=0}$   to $T\iota(S)$ , and by $(\gamma^*(\nabla^{LC;M})_{ {\frac{\p}{\p x}}} E_i)|_{x=0}^\perp$ the projection on the normal bundle. As the Levi-Civita connection $\nabla^{LC;M}$ is torsionless and compatible with the metric $g$ on $M$, we have 
$$
(\gamma^*(\nabla^{LC;M})_{ {\frac{\p}{\p x}}} E_i)|_{x=0}^T=\left(\nabla^{LC;M}_{e_i}u\right)^T=\sum_{i=1}^r g(\nabla^{LC;M}_{e_i}u,e_j) e_j.
$$ 
Therefore
\begin{align*}
\sum_{i=1}^r &\varphi_\chi (e_1, \cdots  , \gamma^*(\nabla^{LC;M})_{ {\frac{\p}{\p x}}} {E}_i\bigr\vert_{x=0}^T, \cdots, e_r,  v, w)\\
&=\left(\sum_{i=1}^r g(\nabla^{LC;M}_{e_i}u,e_i)\right) \varphi_\chi (e_1, \cdots, e_r,  v, w).
\end{align*}
As $u$ is a normal vector field on $\iota(S)$ while $e_i$ is a tangent vector field, we have that $g(u,e_i)$ is identically zero on $S$ and so
by applying $\nabla^{LC;M}_{e_i}$ we find
\[
g(u,(\nabla^{LC;M}_{e_i}e_i)^\perp)=-g(\nabla^{LC;M}_{e_i}u,e_i).
\]
By Lemma \ref{lem:MCV} we therefore get
\[
g (u,W_{[\iota]})=
\sum_{i=1}^{r} g(\nabla^{LC;M}_{e_i}u,e_i).
\]
Equation (\ref{eq:calib2}) then reduces to
\begin{align}\label{eq:calib3}
\varphi_\chi & (\gamma^*(\nabla^{LC;M})_{ {\frac{\p}{\p x}}}\overrightarrow{T_{[\gamma]}S}\bigr\vert_{x=0}, v, w) \\
&\notag =\sum_{i=1}^r \varphi_\chi (e_1, \cdots  , \left(\nabla^{LC;M}_{e_i}u\right)^\perp, \cdots, e_r,  v, w)
\end{align}
and so we have obtained the identity
\begin{equation}\label{eq:last-step}
g(\nabla^\perp _u(JY), w)_{\iota(s)}=\sum_{i=1}^r \varphi_\chi (e_1, \cdots  , \left(\nabla^{LC;M}_{e_i}u\right)^\perp, \cdots, e_r,  v, w)_s
\end{equation}
for ay sections $u,v,w$ of $N_\iota S$ and for any extension $Y$ of $v$ to a vertical vector field along a three-parameter deformation $\gamma$ of $\iota$ by equation (\ref{eq:Exp3}). 
As for fixed $v$ and $w$ the left hand side of equation (\ref{eq:last-step}) only depends on the value of $u$ at $s$, it is not restrictive to assume $\left(\nabla^{LC;M}_{e_i}u\right)_s^\perp=0$ for every $i=1,\dots,r$. Therefore $g(\nabla^\perp _u(JY), w)_{\iota(s)}=0$, as we wanted to prove.
%
%
%
  \end{proof}

Let us recall  the expression of the  Nijenhuis   tensor $N_J$   of  an almost complex structure $J$ on a  smooth  manifold  $\Mf$, see e.g., \cite[p. 123]{KN1969}.
\begin{equation}\label{eq:NT}
N_J(X,Y): =  2\{[JX, JY] - [X,Y] - J[X,JY]- J [JX,Y]\}
\end{equation}
for  $X, Y \in \X(\Mf)$. Moreover, if $\nabla$ is an almost complex affine connection with respect to the almost complex structure $J$
, by  Proposition   3.6 in \cite[p. 145]{KN1969}, we have
\begin{equation}\label{eq:KN1}
N_J(X,Y)=-2\left(T(JX,JY)-J(T(JX,Y))-J(T(X,JY))-T(X,Y)\right),
\end{equation}
where $T$ is the torsion of $\nabla$.  It is noteworthy that the proof  of  Proposition  3.6   ibid.   works not only for  the case   of    finite dimensional   manifolds  $\Mf$, but also for infinite dimensional manifolds for which affine connections and Lie brackets of vector fields can be defined.   Namely, the    argument ibid. uses  the   expression 
for  $T$ defined in (\ref{eq:LC2}), and then  applies it to the RHS of  (\ref{eq:KN1}).   Summing up, the     RHS of (\ref{eq:KN1})  is  written   as  a sum of   eight  summands  involving  $\nabla$   and     four summands   involving  the Lie brackets.  Since  $\nabla$ commute with $J$ by definition of almost complex affine connection,  the sum  of  the  eight  summands  involving  $\nabla$    vanishes.  Then  they observe that  the  sum of the  four summands   involving  the Lie brackets  is   equal to the RHS  of  (\ref{eq:NT}). This   proves  (\ref{eq:KN1}).

In particular, this applies to the infinite dimensional manifold $B^+_{i,f}(S,M)$ As $\nabla^\perp$ is torsion-free by Proposition \ref{prop:pull1}, and it is almost complex with respect to the almost complex structure $J$ by Proposition \ref{prop:almost-complex-connectionf}, we get the following.

\begin{corollary}\label{cor:tor1}  The almost complex structure $J$ on $B^+_{i,f}(S,M)$ satisfies      $N_J = 0$ and so it is a formally integrable almost complex structure.
\end{corollary}

\begin{proposition}\label{prop:almost-complex-connection}
The almost complex structure $J$ is parallel with respect to the Levi-Civita connection $\nabla^{LC}$. Equivalently, $\nabla^{LC}$ is an almost complex connection with respect to the almost complex structure $J$.
\end{proposition}
\begin{proof}
We already   know that  $B^+_{i,f} (S,M)$ is a  smooth manifold  endowed  with the formally integrable   almost complex structure  $J$,  defined by (\ref{eq:J}),  and with  the  $L_2$-metric, which is  a compatible  with $J$ an so defines a  formally K\"ahler structure. 
As the fundamental  2-form $\om$ of the  formally K\"ahler metric  
  is closed,  by \cite[Proposition 4.2, p. 148]{KN1969}\footnote{Kobayshi-Nomidzu gave a  proof  of Proposition 4.2 ibid.  that  is also valid  for   infinite dimensional   manifolds   locally modeled  on  convenient vector spaces,   since  the  standard  formula
  	for $d\om$ ibid.   is  also valid  in  the  convenient  setting
  	\cite[(3), p. 342]{KM1997}.}
  	   we then have
\[
4 \la  (\nabla^{LC}_X J)Y, Z \ra = \la N_J(Y, Z), JX \ra=0
\]
for any three vector fields $X,Y,Z$ on $B^+_{i,f}(S,M)$, hence $\nabla^{LC} J=0$.

\end{proof}

\begin{proof}[Proof of Theorem \ref{thm:formalK}]
	Theorem \ref{thm:formalK}  follows from  Proposition \ref{prop:almost-complex-connection} and 
	Corollary \ref{cor:tor1}.
\end{proof}

\section*{Acknowledgement} 

DF thanks the  Institute  of Mathematics of the Czech Academy of Sciences for hospitality and for the excellent research environment in which this article was started.
DF and HVL  thank     Patrick  Iglesias-Zemmour   and   Ioannis   Chrysikos for     stimulating and helpful  discussions  at the  beginning  of this  project. We  are grateful  to the anonymous   referees  for  critical helpful  comments.

\end{document}